\documentclass[12pt,a4paper]{amsart}

\usepackage{amsmath}
\usepackage{amssymb}
\usepackage{amsthm}
\usepackage{enumitem}
\usepackage[T1]{fontenc}
\usepackage{mathtools}
\usepackage{microtype}
\usepackage[dvipsnames]{xcolor}
\usepackage[all,cmtip]{xy}
\usepackage{url}

\usepackage{hyperref}
\usepackage{amsrefs}
\usepackage[capitalise]{cleveref}

\hypersetup{
    colorlinks=true,
    linkcolor=MidnightBlue,
    citecolor=MidnightBlue,
    urlcolor=MidnightBlue
}

\usepackage[margin=1in]{geometry}
\theoremstyle{plain}
\newtheorem{thm}{Theorem}[section]
\newtheorem{claim}[thm]{Claim}

\newtheorem{proposition}[thm]{Proposition}

\newtheorem{theorem}[thm]{Theorem}

\theoremstyle{definition}

\newtheorem{example}[thm]{Example}
\newtheorem{remark}[thm]{Remark}
\newtheorem*{ack}{Acknowledgements}

\numberwithin{equation}{section}

\DeclareMathOperator{\Aut}{Aut}

\makeatletter
\@namedef{subjclassname@2020}{\textup{2020} Mathematics Subject Classification}
\makeatother

\title{Moishezon manifolds with no nef and big classes}

\author{Jia Jia}
\address{
	National University of Singapore,
	Singapore 119076, Republic of Singapore}
\email{jia\_jia@u.nus.edu}

\author{Sheng Meng}
\address{
	Korea Institute for Advanced Study,
	Seoul 02455, Republic of Korea}
\email{ms@u.nus.edu, shengmeng@kias.re.kr}



\begin{document}

\begin{abstract}
	We show that a compact complex manifold $X$ has no non-trivial nef $(1,1)$-classes
	if there is a non-isomorphic bimeromorphic map
	$f\colon X\dashrightarrow Y$ isomorphic in codimension $1$
	to a compact K\"ahler manifold $Y$ with $h^{1,1}=1$.
	In particular, there exist infinitely many isomorphic classes of smooth compact Moishezon threefolds
	with no nef and big $(1,1)$-classes.
	This contradicts a recent paper
	(Strongly Jordan property and free actions of non-abelian free groups,
	Proc. Edinb. Math. Soc., (2022): 1--11).
\end{abstract}

\maketitle
\tableofcontents

\section{Introduction}

Let \(X\) be a compact complex manifold
with a fixed positive Hermitian form \(\omega\).
Let \(\alpha\) be a closed \((1,1)\)-form.
We use \([\alpha]\) to represent its class in the Bott-Chern \(H^{1,1}_{BC}(X)\).
Recall the following positivity notions
(independent of the choice of \(\omega\)).

\begin{itemize}[leftmargin=2em]
	\item \([\alpha]\) is \emph{K\"ahler} if it contains a K\"ahler form,
	      i.e., if there is a smooth function \(\varphi\) such that
	      \(\alpha+\sqrt{-1}\partial \overline{\partial}\varphi \geq \epsilon \omega\)
	      on \(X\) for some \(\epsilon > 0\).
	\item \([\alpha]\) is \emph{nef} if for every \(\epsilon>0\)
	      there is a smooth function \(\varphi_{\epsilon}\) such that
	      \(\alpha+\sqrt{-1}\partial \overline{\partial}\varphi_{\epsilon} \geq -\epsilon\omega\)
	      on \(X\).
	\item \([\alpha]\) is \emph{big} if it contains a K\"ahler current,
	      i.e., if there exists a quasi-plurisubharmonic
	      function (quasi-psh) \(\varphi\colon X\longrightarrow \mathbb{R} \cup \{-\infty\}\) such that
	      \(\alpha+\sqrt{-1}\partial \overline{\partial}\varphi \geq \epsilon\omega\)
	      holds weakly as currents on \(X\), for some \(\epsilon > 0\).
\end{itemize}

We say $X$ is in \emph{Fujiki's class $\mathcal{C}$} (resp.~\emph{Moishezon})
if it is the meromorphic image of a compact K\"ahler manifold (resp.~projective variety),
or equivalently it is bimeromorphic to a compact K\"ahler manifold (resp.~projective variety).
It is also equivalent to $X$ admitting a big $(1,1)$-class (resp.~big Cartier divisor).
We refer to \cite{Fuj78}*{Definition~1.1 and Lemma~1.1},
\cite{Var89}*{Chapter~IV, Theorem~5} and \cite{DP04}*{Theorem~0.7}
for equivalent definitions and some properties of Fujiki's class $\mathcal{C}$.

Throughout this article,
we work in the Fujiki's class $\mathcal{C}$
where $\partial \overline{\partial}$-lemma holds.
So we are freely to use the equivalent Bott-Chern and de Rham cohomologies.

We start with the following main observation.

\begin{theorem}\label{thm1}
	Let $f\colon X\dashrightarrow Y$ be a bimeromorphic map of compact complex manifolds
	which is isomorphic in codimension \(1\).
	Suppose $X$ is K\"ahler with $h^{1,1}(X,\mathbb{R})=1$ and $f$ is non-isomorphic.
	Then any nef $(1,1)$-class on $Y$ is trivial.
	In particular,
	$Y$ is a non-K\"ahler manifold in Fujiki's class $\mathcal{C}$
	with no nef and big $(1,1)$-classes.
\end{theorem}

One way to construct $f\colon X\dashrightarrow Y$ in \cref{thm1} is
by considering an elementary transformation or a (non-projective) flop.

\begin{example}\label{example1}
	Let $X\subset \mathbb{P}^4$ be a generic smooth quintic threefold.
	By a classical result of Clemens and Katz (cf.~\cites{Cle83,Kat86}),
	$X$ contains a smooth rational curve $C_d$ of degree $d$
	with normal bundle $\mathcal{N}_{C_d/X}\cong \mathcal{O}_{C_d}(-1)^{\oplus 2}$.
	This result was later generalised to a complete intersection of degree $(2,4)$
	in $\mathbb{P}^5$ by Oguiso (cf.~\cite{Ogu94}*{Theorem~2}).
	Let $p\colon Z_d\to X$ be the blowup along $C_d$.
	Then the exceptional divisor $E\cong C_d\times C_d'\cong \mathbb{P}^1\times \mathbb{P}^1$.
	By the contraction theorem of Nakano-Fujiki (cf.~\cite{NF71}),
	there is a bimeromorphic morphism $q\colon Z_d\to Y_d$ to a smooth compact complex manifold $Y_d$
	which contracts $E$ to $C_d'$ along $C_d$.
	Then we can construct $f\coloneqq q\circ p^{-1}\colon X\dashrightarrow Y_d$
	which is isomorphic in codimension \(1\).
	By the Lefschetz hyperplane theorem,
	we see that $h^{2}(X,\mathbb{R})=1$ and hence $h^{1,1}(X,\mathbb{R})=1$.
	Applying \cref{thm1},
	we obtain infinitely many isomorphic classes of smooth Calabi-Yau Moishezon threefolds
	$\{Y_d\}_{d>0}$ satisfying the following theorem.
\end{example}

\begin{theorem}\label{thm2}
	There exist infinitely many isomorphic classes of smooth compact Moishezon threefolds
	with no nef and big $(1,1)$-classes.
\end{theorem}

Nakamura (cf.~\cite{Nak87}*{(3.3)~Remark}) provides another example for the above theorem.

\begin{example}\label{example2}
	There is a bimeromorphic map $f\colon \mathbb{P}^3\dashrightarrow X$
	to a smooth Moishezon threefold $X$ of $h^{1,1}(X,\mathbb{R})=1$
	with no nef and big $(1,1)$-class.
	The map $f$ is constructed by first blowing up a non-singular curve of bidegree $(3,k)$
	with $k\geq 7$ in a smooth quadric surface $S\cong \mathbb{P}^1\times \mathbb{P}^1$
	and then contracting the proper transform of $S$.
	Then $H^{1,1}(X,\mathbb{R})$ is generated by a big divisor $L$ with $L^3<0$.
	So $X$ admits no nef and big class.
	Note that, in this case, $f$ is not isomorphic in codimension $1$.
\end{example}

The aim of the present note is to confute a key theorem
in the recent paper \cite{Kim22} as explained in the following remark.

\begin{remark}
	In \cite{Kim22}*{Theorem~4.2(1)},
	the author asserts that a compact complex manifold $X$ in Fujiki's class $\mathcal{C}$
	always admits a nef and big class.
	However, as we just discussed,
	\cref{example1,example2} or \cref{thm2} confute this claim.
	Note that \cite{Kim22}*{Theorem~4.2(1)} plays a crucial role
	in the proof of \cite{Kim22}*{Corollary~4.3} that $\Aut_{\tau}(X)/\Aut_0(X)$ is finite
	where $\Aut_{\tau}(X)$ is the group of automorphisms (pullback)
	acting trivially on $H^2(X,\mathbb{R})$
	and $\Aut_0(X)$ is the neutral component.
	So the proof there does not work.
	Nevertheless, the statement \cite{Kim22}*{Corollary~4.3} still holds
	and was previously proved by showing the existence of equivariant K\"ahler model;
	see \cite{JM22}*{Theorem~1.1, Corollary~1.3}.
\end{remark}

It is known that a smooth compact surface in Fujiki's class \(\mathcal{C}\) is K\"ahler
and hence a smooth Moishezon surface is projective.
So \cref{thm2} is optimal in terms of minimal dimension
and it is easy to construct examples,
like those in \cref{thm2},
of arbitrary higher dimensions by further taking the product
with a smooth projective variety of suitable dimension.
In the singular surface case, we summarize several examples constructed
by Schr\"oer (cf.~\cite{Sch99}) and Mondal (cf.~\cite{Mon16}) in the following remark.

\begin{remark}
	The examples in \cite{Sch99} are constructed in a similar way
	by different elementary transformations of $\mathbb{P}^1\times C$ where the genus $g(C)>0$.
	However, they behave quite differently on Cartier divisors.
	The example in \cite{Mon16}*{\S~2} is a supplement to (1) on the rational case.
	It seems that we do not know any rational example for point (3).
	\begin{enumerate}[leftmargin=2em]
		\item (cf.~\cite{Sch99}*{\S~3})
		      There is a non-projective normal compact Moishezon surface $S$
		      such that the Picard number of $S$ is $0$.
		      In particular, $S$ admits no non-trivial nef Cartier divisor.
		\item (cf.~\cite{Mon16}*{\S~2})
		      There is a non-projective normal compact Moishezon \textbf{rational} surface $S$
		      such that the Picard number of $S$ is $0$.
		      The surface is $Y_2'$ in \cite{Mon16}*{\S~2}.
		      We give some explanation on the Picard number.
		      Note that the Weil-Picard number of $S$ is $1$
		      (cf.~\cite{Nak17}*{Definition~2.7 and Lemma~2.10}).
		      Since $S$ is not projective
		      (cf.~\cite{Mon16}*{Theorem~4.1 and Example~3.19}),
		      the Picard number of $S$ has to be $0$
		      (cf.~\cite{Nak17}*{Definition~2.11--Remark~2.13, Remark on the top of page 303}).
		\item (cf.~\cite{Sch99}*{\S~4})
		      There is a non-projective normal compact Moishezon surface $Z$
		      which allows a non-projective birational morphism $Z\to S$ to a projective surface $S$.
		      In particular, $Z$ admits a nef and big Cartier divisor
		      which is the pullback of an ample Cartier divisor on \(S\).
	\end{enumerate}
\end{remark}

\section{Proof of Theorem~\ref{thm1}}

We first reprove \cite{Gol21}*{Theorem~4.5} by the following \cref{prop1}.
The first version of this proposition was formulated in \cite{Fuj81}*{Corollary~3.3}
where Fujiki works in the smooth setting and $f_*[\alpha]$ is assumed to be semi-positive.
Later, it was generalized by Huybrechts (cf.~\cite{Huy03}*{Proposition~2.1})
to the situation when canonical bundles $K_X$ and $K_Y$ are nef
and $[\alpha]$ and $f_*[\alpha]$ are only assumed
to have positive intersections with all rational curves.

When dealing with the singular setting in the below \cref{prop1},
we refer to \cite{HP16} for the basic definitions involved.
Note that for a normal compact complex space \(X\) with rational singularities,
\(H^{1,1}(X,\mathbb{R})\) embeds into \(H^2(X,\mathbb{R})\) naturally,
and the intersection product on \(H^{1,1}(X,\mathbb{R})\)
can be defined via the cup-product for \(H^2(X,\mathbb{R})\)
(cf.~\cite{HP16}*{Remark~3.7}).
Of course, for the purpose of this note, one can focus on the smooth setting for simplicity.

\begin{proposition}[cf.~\cite{Gol21}*{Theorem~4.5} and \cref{remark1}]\label{prop1}
	Let $f\colon X\dashrightarrow Y$ be a bimeromorphic map of normal compact complex spaces
	with rational singularities.
	Suppose $f$ does not contract divisors
	and there exists a K\"ahler class $[\alpha]\in H^{1,1}(X,\mathbb{R})$ such that $f_*[\alpha]$ is nef.
	Then $f^{-1}$ is holomorphic.
\end{proposition}

\begin{proof}
	Consider the log resolution of the indeterminacy of $f$:
	\[
		\xymatrix{
			{} & Z \ar[ld]_{p} \ar[rd]^{q} \\
			X \ar@{.>}[rr]^{f} & {} & Y
		}
	\]
	where $p\colon Z\to X$ and $q\colon Z\to Y$ are the two projections.
	By Chow's lemma (cf.~\cite{Hir75}*{Corollary~2 and Definition~4.1}),
	we may assume $p$ is a projective morphism
	obtained by a finite sequence of blowups along smooth centres.
	Denote by $\bigcup_{i=1}^n E_i$ the full union of exceptional prime divisors of $p$.
	One can find
	\[
		E= \sum_{i=1}^n \delta_i E_i
	\]
	with suitable $\delta_1,\ldots,\delta_n>0$
	such that $-E$ is $p$-ample (cf.~\cite{DP04}*{Proof of Lemma~3.5}).
	Here, if $n=0$, then $p$ is isomorphic and $-E=0$ is automatically $p$-ample.

	Note that $p^*[\alpha]$ is represented by a smooth semi-positive form
	and $q$-exceptional divisors are also $p$-exceptional divisors
	since $f$ does not contract divisors by the assumption.
	Applying \cite{Gol21}*{Lemma~4.4} (cf.~\cite{Fuj81}*{Lemma~2.4}) to $p^*[\alpha]$,
	we have
	\[
		q^*q_*p^*[\alpha]-p^*[\alpha]=\sum_{i=1}^n a_i [E_i]
	\]
	with $a_i\geq 0$.

	\begin{claim}\label{claim-T}
		We claim that $q^*q_*p^*[\alpha]-p^*[\alpha]=0$.
	\end{claim}

	\begin{proof}
		Suppose the contrary that $a_1>0$ without loss of generality.
		Note that $q^*q_*p^*[\alpha]$ is nef and $p^*[\alpha]$ is $p$-trivial.
		Then the divisor
		\[
			D\coloneqq \sum_{i=1}^n a_i E_i-\epsilon E=\sum_{i=1}^n (a_i-\epsilon\delta_i) E_i
		\]
		is $p$-ample and $-D$ is not effective whenever $0<\epsilon<a_1/\delta_1$.
		We can further find rational coefficients $b_i$ sufficiently closed to $a_i-\epsilon\delta_i$
		such that
		\[
			D'\coloneqq \sum_{i=1}^n b_i E_i
		\]
		is still $p$-ample and $-D'$ is not effective.
		Note that $mD'$ is then a Cartier divisor for a suitable integer $m$ and $p_*(-mD')=0$.
		By the negativity lemma for Cartier divisors (cf.~\cite{Wan21}*{Lemma~1.3}),
		$-mD'$ is effective, a contradiction.
		So the claim is proved.
	\end{proof}

	Applying Chow's lemma again,
	there is a bimeromorphic morphism $\sigma\colon W\to Z$
	such that $q\circ\sigma$ is a projective morphism.
	Note that $(q\circ\sigma)_*(p\circ\sigma)^*[\alpha]=q_*p^*[\alpha]$.
	So we may replace $Z$ by $W$
	and assume $q$ is already projective (without requiring $p$ being projective).
	Let $F$ be any fibre of $q$ which is projective.
	Let $C$ be any curve in $F$.
	By the projection formula and \cref{claim-T},
	\[
		\int \alpha\wedge \langle p_*C\rangle
		=\int p^*\alpha\wedge \langle C \rangle
		=\int q^*q_*p^*\alpha\wedge \langle C \rangle
		=\int q_*p^*\alpha\wedge \langle q_*C \rangle=0
	\]
	where $\langle-\rangle$ represents the integration current.
	Since $[\alpha]$ is K\"ahler, $p(C)$ is a point and hence $p(F)$ is a point.
	By the rigidity lemma (cf.~\cite{Gol21}*{Lemma~4.1})
	which is essentially due to the Riemann extension theorem (cf.~\cite{GR84}*{Page~144}),
	$f^{-1}\colon Y\to X$ is a holomorphic map.
\end{proof}

\begin{remark}\label{remark1}
	\Cref{claim-T} was treated in the proof of \cite{Gol21}*{Theorem~4.5, Equation~(4.4)}.
	However, the proof there seems incomplete after Equation (4.2) where the author
	claims ``the singular locus of the nef class is empty''.
	This is also mentioned after \cite{Gol21}*{Definition~4.3}
	where the author seems to have misinterpreted a result of Boucksom.
	Note that a nef class has empty singular locus if and only if it is semi-positive.
	However, there are situations where non-semi-positive nef classes exist.
	Nevertheless, we can overcome this gap by applying the negativity lemma
	as in the proof of \cref{claim-T}.
\end{remark}

\begin{proof}[Proof of \cref{thm1}]
	Note that $h^{1,1}(Y,\mathbb{R})=h^{1,1}(X,\mathbb{R})=1$
	because $f$ is isomorphic in codimension \(1\) (cf.~\cite{RYY19}*{Corollary~1.5}).
	Let $[\alpha]$ be a K\"ahler class on $X$.
	Then $H^{1,1}(Y,\mathbb{R})$ is generated by the big class $f_*[\alpha]$.
	Let $[\gamma]\in H^{1,1}(Y,\mathbb{R})$ be a nef class.
	Then $[\gamma]=tf_*[\alpha]$ for some $t\geq 0$ (cf.~\cite{Fuj81}*{Lemma~2.1}).
	So it suffices to show that $f_*[\alpha]$ is not nef.

	Suppose the contrary that
	$f_*[\alpha]\in H^{1,1}(Y,\mathbb{R})$ is nef.
	By \cref{prop1}, $f^{-1}$ is holomorphic.
	By the purity (cf.~\cite{GR55}*{Satz~4})
	and since $f^{-1}$ is isomorphic in codimension \(1\),
	the exceptional locus of $f^{-1}$ is empty.
	In particular, $f$ is isomorphic, a contradiction.
\end{proof}

\begin{ack}
	The authors would like to thank Professor~Boucksom
	and Professor~Oguiso for the valuable discussion and suggestions.
	The first author is supported by a President's Scholarship of NUS\@.
	The second author is supported by a Research Fellowship of KIAS (MG075501).
\end{ack}


\end{document}